\documentclass[11pt]{amsart}

\usepackage[utf8]{inputenc}
\usepackage[english]{babel}
\usepackage[margin=1in]{geometry}
\usepackage{microtype}
\usepackage{amsmath,amsthm,amssymb,mathrsfs}
\usepackage{mathtools}
\usepackage{graphicx}
\usepackage{hyperref}
\usepackage{tikz-cd,adjustbox}
\usepackage{spverbatim,comment}
\usepackage[shortlabels]{enumitem}
\usepackage{tabularx}
\usepackage{colonequals}
\usepackage{caption}

\theoremstyle{plain}
\newtheorem{theorem}{Theorem}[section]
\newtheorem{lemma}[theorem]{Lemma}
\newtheorem{proposition}[theorem]{Proposition}
\newtheorem{corollary}[theorem]{Corollary}

\theoremstyle{definition}
\newtheorem{definition}[theorem]{Definition}

\newtheorem{example}[theorem]{Example}
\theoremstyle{remark}
\newtheorem{remark}[theorem]{Remark}

\newcommand{\isom}{\cong} 

\newcommand{\R}{\mathbb{R}}

\newcommand{\F}{\mathbb{F}}
\newcommand{\C}{\mathbb{C}}
\newcommand{\Z}{\mathbb{Z}}

\renewcommand{\P}{\mathbb{P}}
\newcommand{\Q}{\mathbb{Q}}

\DeclareMathOperator{\PGL}{PGL}

\newcommand*{\sHom}{\mathscr{H}\kern -.5pt om}

\DeclareMathOperator{\rk}{rk}
\DeclareMathOperator{\Sing}{Sing}

\begin{document}
\title{Endomorphisms of rank one Gorenstein del Pezzo surfaces}

\author{Rohan Joshi}
\address{UCLA Department of Mathematics, Los Angeles, CA 90095-1413}
\email{rohansjoshi@math.ucla.edu}


\begin{abstract}
We prove that, in all except one case, a Gorenstein del Pezzo surface of
Picard rank $1$ admits an int-amplified endomorphism if and only if it
is a quotient of a toric variety by a finite group which acts freely in
codimension one and preserves the open torus. We classify all such
quotients.
\end{abstract}
\maketitle



\section{Introduction}


A surjective endomorphism of a projective algebraic variety $f\colon X
\to X$ is said to be \textit{int-amplified} if there exists an ample
Cartier divisor $A$ on $X$ such that $f^*A - A$ is ample. Equivalently,
$f$ is int-amplified if all eigenvalues of the operator $f^* \colon
N^1(X)_{\C} \to N^1(X)_{\C}$ have magnitude greater than $1$
\cite[Theorem 1.1]{Meng-Building-Blocks}, where $N^1(X)$ is the group
of divisors on $X$ up to numerical equivalence. Roughly speaking, an
endomorphism is int-amplified if it expands the algebraic variety ``in
all directions.''

Whether a given variety admits any int-amplified endomorphism is an
interesting question of algebraic dynamics and seems to be a rare
property. Abelian varieties admit int-amplified endomorphisms (e.g. the
multiplication-by-$m$ endomorphism) and normal toric varieties admit
int-amplified endomorphisms (e.g. the $m$-th power Frobenius morphism).
Meng and Zhong conjecture \cite[Question 1.2]{Meng-Rigidity} that a
smooth projective rationally connected complex variety admits an
int-amplified endomorphism if and only if it is a toric variety (see
also \cite[Question 4.4]{Fakhruddin}). This has been proved for smooth
surfaces \cite[Theorem 3]{Nakayama02} and smooth Fano threefolds
\cite[Theorem 1.4]{Meng-Zhang-Zhong}, \cite[Theorem
6.1]{Totaro-Log-Bott}. 

When we consider singular Fano varieties, we find examples that admit
int-amplified endomorphisms which are not toric, but are finite
quotients of toric varieties.
\begin{example}\label{quaternion-example}
Consider the two-dimensional representation of the quaternion group
given by
\[ 
   i\mapsto \begin{pmatrix} i&0\\0&-i\end{pmatrix},\quad
   j\mapsto \begin{pmatrix}0&1\\-1&0\end{pmatrix},\quad
   k=ij\mapsto \begin{pmatrix}0&i\\i&0\end{pmatrix}.
\]
Extending this to a projective representation in $\PGL(3)$, we can
consider it as an action of $Q_8$ on $\P^2$. Let $X$ be the quotient of
$\P^2$ by this action. Then the image of $[0:0:1]$ is a $D_4$
singularity: $D_4$ singularities are not toric so $X$ cannot be a toric
variety. Nevertheless we can easily see that the ``Frobenius
endomorphism'' $[x:y:z] \mapsto [x^5:y^5:z^5]$ of $\P^2$ commutes with
the action of $Q_8$ and thus descends to an int-amplified endomorphism
of $X$. 
\end{example}

In view of such examples, it is a folklore conjecture that a klt Fano
variety admits an int-amplified endomorphism if and only if it a finite
quotient of a toric variety. Nakayama proved \cite[Lemma
2.6]{NakayamaStruct} that a quotient of a toric surface by a finite
group which acts freely in codimension one and preserves the open torus
inherits an int-amplified endomorphism from the toric variety. Nakayama
also proved that for a klt del Pezzo surface $X$ of Picard rank $>1$,
$X$ admits an int-amplified endomorphism iff $X$ is a quotient of a
toric variety by a finite group acting freely in codimension $1$ such
that the action preserves the open torus \cite[Theorem 1.3]{Nakayama}. 

In this article, I study the case of Picard rank $1$ and consider del
Pezzo surfaces with Gorenstein singularities. Note that an
endomorphism of a projective variety of Picard rank $1$ is int-amplified
if and only if it has degree greater than one, since in this case
$N^1(X)_{\C} \isom \C$ and $f^*: N^1(X)_{\C} \to N^1(X)_{\C}$ is just
multiplication by $(\deg f)^{\dim X}$. Define $S'(E_8)$ as the weighted
sextic surface $\{ X^5Y + X^4Z + Z^3+ W^2 = 0 \} \subset \P(1, 1, 2, 3)$
\cite{Gurjar}: it is a Gorenstein del Pezzo surface of Picard rank 1.
Then the main theorem of this paper is the following.

\begin{theorem}\label{main_theorem} Let $X$ be a Gorenstein del Pezzo
surface of Picard rank $1$ over an algebraically closed field of
characteric zero which is not isomorphic to the surface $S'(E_8)$. Then
$X$ admits an int-amplified endomorphism (i.e. an endomorphism of degree
$>1$) if and only if $X$ is a quotient of a toric variety by a finite
group acting freely in codimension one that preserves the open torus.
\end{theorem}

We give an overview of rank $1$ Gorenstein del Pezzo surfaces below and
their classification by singularity type. To prove Theorem
\ref{main_theorem}, we first identify which rank $1$ Gorenstein del
Pezzo surfaces are quotients of toric surfaces. The ``easy direction''
(that is, the statement that if $X$ is such a quotient of a toric
variety, it admits an int-amplified endomorphism) follows from
\cite[Lemma 2.6]{NakayamaStruct}. For the rest, following Kawakami and
Totaro \cite{Kawakami-Totaro}, we use the tool of Bott vanishing,
combined with lifting to universal covers, to deduce that they do not
admit int-amplified endomorphisms. To find failures of Bott vanishing,
we use a version of the Riemann-Roch formula for normal surfaces. This
formula is automated in a Python program which is in the accompanying
code repository \cite{CodeRepo}.

We work in characteristic zero, but everything in this paper should
hold over an algebraically closed field of characteristic $\neq 2, 3$,
as long as we everywhere replace ``int-amplified endomorphism'' with
``int-amplified endomorphism of degree invertible in the base field''.
The classification of rank $1$ Gorenstein del Pezzo surfaces is the same
over algebraically closed fields of characteristic $\neq 2, 3$
\cite[Theorem B.6]{Lacini}, and all the endomorphisms described in this
paper have degree coprime to $2$ and $3$.

\subsection*{Acknowledgments} I would like to thank Burt Totaro for many
helpful conversations. The author is supported by NSF Graduate Research
Fellowship Grant No. DGE-2034835.

\section{Preliminaries}

A \textit{(log) del Pezzo surface} is a projective klt surface $X$ such
that $-K_X$ is ample. A del Pezzo surface $X$ is \textit{Gorenstein} if
$-K_X$ is Cartier. A del Pezzo surface is Gorenstein if and only if it
has ADE (a.k.a Du Val) singularities, if and only if it has canonical
singularities. The degree of a normal surface $X$ is the intersection
number $(-K_X)^2$. The degree of a Gorenstein del Pezzo surface $d$
satisfies $1 \leq d \leq 9$.

\begin{theorem}\cite[Theorem B.6]{Lacini} Let $X$ be a Gorenstein del
Pezzo surface of Picard rank $1$. Then $X$ is either isomorphic to
$\P^2$ or has one of the singularity types listed below. The surface $X$
is determined up to isomorphism by its singularity type, except in the
cases $E_8$, $A_1+E_7$ and $A_2+E_6$, which have two isomorphism classes
each, and the case $2D_4$, which has infinitely many isomorphism
classes.
\[
\begin{array}{cccccc}
A_1 & A_1+A_2 & A_4 & 2A_1+A_3 & D_5 & A_1+A_5 \\
3A_2 & E_6 & 3A_1+D_4 & A_7 & A_1+D_6 & E_7 \\
A_1+2A_3 & A_2+A_5 & D_8 & 2A_1+D_6 & E_8 & A_1+E_7 \\
A_1+A_7 & 2A_4 & A_8 & A_1+A_2+A_5 & A_2+E_6 & A_3+D_5 \\
4A_2 & 2A_1+2A_3 & 2D_4 & & &
\end{array}
\]
\end{theorem}

Here singularity type $2A_1+A_3$, for example, means that the surface has 
three singular points: two $A_1$ singularities and one $A_3$ singularity.

There are $28$ families of Gorenstein del Pezzo surfaces of rank $1$:
$24$ consisting of a single isomorphism class, $3$ consisting of two
isomorphism classes, and one one-parameter family. For all families
consisting of a single isomorphism class, we denote a representative of
that isomorphism class by $S(A_1)$, etc. For families with two
isomorphism classes (i.e. the singularity types $E_8$, $A_1+E_7$, and
$A_2+E_6$), we denote representatives of the two isomorphism classes of
surfaces with that singularity type by $S(E_8)$ and $S'(E_8)$, etc. In
these cases, to distinguish them, $S$ denotes the ``special element''
and $S'$ a ``general element'' of the family. For example, we can define
$S(E_8) \colonequals \{ X^5Y + Z^3+ W^2 = 0 \} \subset \P(1, 1, 2, 3)$
and $S'(E_8) \colonequals \{ X^5Y + X^4Z + Z^3 + W^2 = 0 \} \subset
\P(1, 1, 2, 3)$, and the family is given by $\{ X^5Y + cX^4Z + Z^3 + W^2
= 0 \} \subset \P(1, 1, 2, 3)$ \cite{Gurjar} (see also \cite{DuVal} for
other explicit equations for Gorenstein del Pezzo surfaces). Finally, we
denote the surfaces in the one one-parameter family by $S_t(2D_4)$. 

The minimal resolution of a Gorenstein del Pezzo surface is a
\textit{smooth weak del Pezzo surface} (that is, a smooth surface with
$-K_X$ nef and big). A smooth weak del Pezzo surface is either a blowup
of $\P^2$ at most eight times, or is isomorphic to $\F_0 = \P^1 \times
\P^1$ or $\F_2 = \P_{\P^1}(\mathcal{O} \oplus \mathcal{O}(2))$. If $X$
is a smooth weak del Pezzo surface and $C \subset X$ is a curve, then
since $-K_X.C \geq 0$, the adjunction formula implies the
self-intersection of $C$ is at least $-2$. A $(-1)$- (resp. $(-2)$-)
curve on a surface is a smooth rational curve $C$ with self-intersection
$-1$ (resp. $-2$). Martin and Stadlmayr \cite{Martin-Stadlmayr}
classifies (smooth) weak del Pezzo surfaces that have a global vector
field, and for each draws the arrangement of $(-1)$- and
$(-2)$-curves. Miyanishi and Zhang \cite{Miyanishi-Zhang} draws the
arrangement of $(-1)$- and $(-2)$-curves on the minimal resolutions of
all rank $1$ Gorenstein del Pezzo surfaces of degree $1$. Going in the
other direction, every Gorenstein del Pezzo surface can be obtained as
the anticanonical model of a smooth weak del Pezzo surface, by
contracting all $(-2)$-curves on it. Thus there is a bijection between
isomorphism classes of smooth weak del Pezzo surfaces and Gorenstein del
Pezzo surfaces.

The most important tool we will use in the study of int-amplified
endomorphisms of Gorenstein del Pezzo surfaces is Bott vanishing.

\begin{definition}
Let $X$ be a normal projective variety. $X$ is said to satisfy
\textit{Bott vanishing} (for Weil divisors) if 
\[
H^i(X, \Omega^{[j]}_X(D)) = 0
\]
for all ample Weil divisors $D$, $i > 0$, $j \geq 0$. A Weil divisor $D$ is
ample if some positive multiple $mD$ is an ample Cartier divisor. By
$\Omega^{[j]}(D)$ we mean the reflexive sheaf $(\Omega^j_X \otimes
\mathcal{O}(D))^{**}$. If the vanishing only holds for all $D$ an ample
Cartier divisor, we say $X$ satisfies \textit{Bott vanishing for Cartier
divisors}.
\end{definition}

\begin{theorem}\cite[Theorem C]{Kawakami-Totaro} Let $X$ be a normal
projective variety over a perfect field $k$. If $X$ admits an
int-amplified endomorphism of degree invertible in $k$, then $X$
satisfies Bott vanishing.
\end{theorem}





Baker in his Ph.D. thesis \cite{Baker} studied Bott vanishing for
Gorenstein del Pezzo surfaces using the Leray spectral sequence and
classified those that satisfy Bott vanishing for Cartier divisors.

\begin{theorem}\cite[p.14]{Baker}\label{BakerThm} Let $X$ be a Gorenstein del
Pezzo surface and $\widetilde{X}$ its minimal resolution. Let $n$ be the number
of $(-2)$-curves on $\widetilde{X}$. Then $X$ satisfies Bott vanishing for
Cartier divisors if and only if \[ h^0(\widetilde{X}, T_{\widetilde{X}}) =
10-n-2\rho(X). \]
\end{theorem}

\begin{corollary}
Let $X$ be a rank $1$ Gorenstein del Pezzo surface. Then $X$ satisfies
Bott vanishing for Cartier divisors iff it is not one of
the following surfaces: $S(E_8)$, $S(A_1+E_7)$, $S(A_2+E_6)$, or
$S_t(2D_4)$.
\end{corollary}
\begin{proof}
See \cite[Proposition 2.7.2]{Baker}
\end{proof}

There are finitely many families of Gorenstein del Pezzo surfaces: a
complete classification can be found in \cite{Dolgachev}. Baker's thesis
contains tables which indicate which families satisfy/fail Bott
vanishing for Cartier divisors. On page 14 of \cite{Baker}, there are
two tables: the first table describes Gorenstein del Pezzo surfaces of
degree $>1$ whose minimal resolution has a global vector field (these
weak del Pezzo surfaces are classified in \cite{Martin-Stadlmayr}), and
the second describes those whose minimal resolution does not have a
global vector field. The second table has some gaps which we correct
here. We have:

\begin{center}
\begin{tabular}{p{1.5cm} p{7cm} p{7cm}}
\hline
\textbf{Degree} & \textbf{Satisfies Cartier BV} & \textbf{Fails Cartier
BV} \\
\hline
5 & $\emptyset$ &  \\
4 & $2A_1$, $A_2$ & $A_1$, $\emptyset$ \\
3 & $D_4$, $A_4$, $A_3+A_1$, $A_2+2A_1$, $4A_1$ & $A_3$, $A_2+A_1$,
$3A_1$, $A_2$, $2A_1$, $A_1$, $\emptyset$ \\
2 & $E_6$, $D_6$, $D_5+A_1$, $D_4+2A_1$, $A_6$, $A_5+A_1$, $A_4+A_2$,
 $2A_3$, $A_3+A_2+A_1$, $A_3+3A_1$, $3A_2$, $6A_1$ & $D_5$, $D_4+A_1$, 
 $A_5$, $A_3+A_2$, $2A_2+A_1$, $A_4+A_1$, $A_3+2A_1$, $A_2+3A_1$,
 $5A_1$, $D_4$, $2A_2$, $A_4$, $A_3+A_1$, $A_2+2A_1$, $4A_1$, $A_3$,
 $A_2+A_1$, $3A_1$, $A_2$, $2A_1$, $A_1$, $\emptyset$ \\
\hline
\end{tabular}
\captionof{table}{Gorenstein del Pezzo surfaces $X$ with $(-K_X)^2 > 1$
and $h^0(\widetilde{X}, T_{\widetilde{X}}) = 0$}
\hfill
\end{center}
\hfill

However, failure of Bott vanishing for Cartier divisors is a weaker tool
than Bott vanishing for Weil divisors for the purpose of studying
endomorphisms. In the case of rank $1$ Gorenstein del Pezzo surfaces,
three of the 30 rigid Gorenstein del Pezzo surfaces fail Bott vanishing
for Cartier divisors (in addition to all surfaces $S_t(D_4)$), but at
least fourteen fail Bott vanishing (for Weil divisors).


\section{Quotients of toric surfaces}


\begin{proposition}\label{toric} The five surfaces $\P^2$, $S(A_1)$,
$S(A_1+A_2)$, $S(2A_1+A_3)$ and $S(3A_2)$ are toric, and thus admit
int-amplified endomorphisms.
\end{proposition}
\begin{proof}
$\P^2$ is of course toric, and $S(A_1)$ is the projective closure of the
quadric cone, which is a toric variety: it can also be described as the
weighted projective space $\P(1, 1, 2)$. $S(A_1+A_2)$ can be described
as the weighted projective space $\P(1, 2, 3)$. It is easy to see these
three surfaces are precisely the weighted projective planes with at
worst canonical singularities. $S(2A_1+A_3)$ and $S(3A_2)$ are given
respectively by the following complete fans. \\

\begin{minipage}{0.45\textwidth}
\begin{center}
\begin{tikzpicture}[scale=0.8]
    \foreach \x in {-2,...,2} {
        \foreach \y in {-2,...,2} {
            \fill[black] (\x,\y) circle (1pt);
        }
    }
    \draw[->, thick] (0,0) -- (1,2);
    \draw[->, thick] (0,0) -- (1,-2);
    \draw[->, thick] (0,0) -- (-1,0);
\end{tikzpicture}
\end{center}
\end{minipage}
\hfill
\begin{minipage}{0.45\textwidth}
\begin{center}
\begin{tikzpicture}[scale=0.8]
    \foreach \x in {-2,...,2} {
        \foreach \y in {-2,...,2} {
            \fill[black] (\x,\y) circle (1pt);
        }
    }
    \draw[->, thick] (0,0) -- (-2,1);
    \draw[->, thick] (0,0) -- (1,1);
    \draw[->, thick] (0,0) -- (1,-2);
\end{tikzpicture}
\end{center}
\end{minipage}

\end{proof}

\begin{remark}
We can also describe $S(2A_1+A_3)$ as a quotient of $\P^2$. Consider the
two dimensional linear representation of $\Z/4\Z$ where the generator
acts by the matrix
\[
\begin{bmatrix}
i & 0 \\
0 & -i
\end{bmatrix}
\]
Then we can extend this to a projective representation in $\PGL(3)$ and
consider it as an action of $\Z/4\Z$ on $\P^2$. Define $X$ as the 
quotient of $\P^2$ by this action. The image of the origin $[0:0:1]$ in
$X$ is an $A_3$ singularity. The square of the generator (but not the
generator) acts as a quasi-reflection locally around the two points
$[1:0:0]$ and $[0:1:0]$, so the images of $[1:0:0]$ and $[0:1:0]$ are
$A_1$ singularities. There are no other singular points on $X$, so $X
\isom S(2A_1+A_3)$.
\end{remark}

\begin{remark}
Similarly, we can describe $S(3A_2)$ as a quotient of $\P^2$. Consider
the representation of $\Z/3\Z$ where the generator acts by the matrix
\[
\begin{bmatrix}
\omega & 0 \\
0 & \omega^2
\end{bmatrix}
\]
where $\omega = e^{2\pi i /3}$. Extend this to an action of $\Z/3\Z$ on
$\P^2$ (so the generator sends $[x:y:z] \mapsto [\omega x: \omega^2 y:
z]$) and let $X$ be the quotient of $\P^2$ by this action. We can see
that the images of $[0:0:1]$, $[1:0:0]$ and $[0:1:0]$ are $A_2$
singularities and there are no other singular points. So $X \isom
S(3A_2)$.
\end{remark}

These five are the only toric rank one Gorenstein del Pezzo surfaces, by
the classification of toric Gorenstein del Pezzo surfaces
\cite{ToricGDP}.

\begin{proposition}\label{toric_quotient} The five surfaces
$S(3A_1+D_4)$, $S(A_1+2A_3)$, $S(A_1+A_2+A_5)$, $S(4A_2)$ and
$S(2A_1+2A_3)$ are quotients of toric surfaces by a finite group that
acts freely in codimension one and preserves the open torus.
\end{proposition}
\begin{proof}
We describe each of these surfaces as a quotient of a toric variety, and
find an int-amplified endomorphism of the toric variety that descends to
the quotient (we can also apply \cite[Lemma 2.6]{NakayamaStruct} to show
that the quotients admit int-amplified endomorphisms).
\begin{itemize}
\item $S(3A_1+D_4)$ is the surface constructed in Example
\ref{quaternion-example}.
   
\item $S(A_1+2A_3)$  is the quotient of $\P^1 \times \P^1$ by $\Z/4$,
where the action is generated by $([x:y],[u:v]) \mapsto ([u:v],[x:-y])$.
The morphism raising each coordinate in each $\P^1$ to the fifth power
descends to the quotient. 

\item $S(A_1+A_2+A_5)$ can be described as the quotient of the action
$\Z/6$ on the sextic del Pezzo surface. Let $X$ be the blowup of $\P^2$
at the three toric points $[1:0:0]$, $[0:1:0]$ and $[0:0:1]$. Then the
Cremona automorphism and the order $3$ automorphism given in coordinates
by $[x:y:z] \mapsto [y:z:x]$ (lifted from $\P^2$) generate a cyclic
group of automorphisms of $X$. The quotient of $X$ by this automorphism
is $S(A_1+A_2+A_5)$, and all $m$-th power endomorphisms of $X$ commute
with the automorphisms and descend to the quotient.

\item $S(4A_2)$ is $\P^2$ mod the action of $\Z/3 \times \Z/3$ where the
two generators act on $\P^2$ by $[x:y:z] \mapsto [x:\omega y: \omega^2
z]$ and $[x:y:z] \mapsto [y:z:x]$, respectively. The seventh power
Frobenius morphism descends to the quotient.

\item $S(2A_1+2A_3)$ is $\P^1 \times \P^1$ modulo the action of $\Z/2
\times \Z/4$, where the generators act by $([x,y],[u:v]) \mapsto
([x:-y],[u:-v])$ and $([x:y],[u:v]) \mapsto ([u:v],[y:-x])$. The fifth
power morphism descends to the quotient.
\end{itemize}
In all of these examples the group action preserves the open torus. It
is routine to check that the quotients have the stated singularity type. 
\end{proof}


\section{Lifting endomorphisms}

\begin{definition}
	Let $Y$ be a normal variety such that the fundamental group of $Y^0
	= Y -\Sing Y$ is finite. Define $X$ to be the normalization of $Y$
	in the function field of the universal cover of $Y^0$. Then there is
	finite morphism $\pi: X \to Y$ which is étale over $Y^0$. We call
	$X$ the \textit{quasi-universal cover} of $Y$.

\end{definition}

The fundamental group of the smooth locus of a complex Fano variety is
finite \cite{Braun}, so it has a quasi-universal cover. Miyanishi and
Zhang \cite[Lemma 4]{Miyanishi-Zhang} prove that the quasi-universal
cover of a Gorenstein del Pezzo surface is also a Gorenstein del Pezzo
surface.

\begin{proposition}\label{lifting_lemma} Let $Y$ be a normal surface and
$X$ its quasi-universal cover. Then any int-amplified endomorphism
$f\colon  Y \to Y$ lifts to an int-amplified endomorphism of $X$.
\end{proposition}

\begin{proof}
Let $\phi: X \to Y$ be the covering map and set $V = (f \circ
\phi)^{-1}(Y^0) \subset X$. Since $f$ has finite fibers, $V$ is
$\phi^{-1}(Y^0)$ minus a finite set of points. Since $\phi^{-1}(Y^0)$ is
smooth and simply-connected, $V$ is also simply-connected. Thus there
exists a lift
\begin{center}
\begin{tikzcd}
 & & X \arrow[d, "\phi"] \\
V \arrow[urr, bend left=12] \arrow[r, "\phi|_V"] & Y \arrow[r, "f"] & Y
\end{tikzcd}
\end{center}
We can think of this as a rational map $\overline{f}: X \dashrightarrow
X$ which lifts $f$. I claim $\overline{f}$ is defined on all of $X$.
Assume for the sake of contradiction that the indeterminacy locus $S$ of
$f$ is a nonempty set of points. Then we can blow up $X$ finitely many
times to obtain a variety $B$, with $p: B \to X$ the blowup morphism,
such that $\overline{f}$ extends to a morphism $\widetilde{f}: B \to X$. If
$C$ is an exceptional curve in the fiber of some $x \in S$, then
$\widetilde{f}(C)$ must be a curve (otherwise $\overline{f}$ would be
determined at $x$). But since $\overline{f}$ is a lift of $f$, 
\begin{center}
\begin{tikzcd}
B \arrow[r, "\widetilde{f}"] \arrow[d, "\phi \circ p"] & X \arrow[d, "\phi"] \\
Y \arrow[r, "f"] & Y 
\end{tikzcd}
\end{center}
commutes. The $\phi$ does not contract any curves, so $\phi(\widetilde{f}
(C))$ is a curve, but this is impossible since $p(C) = x$ which is a
point. 

Thus $\overline{f}: X \to X$ is a morphism which lifts $f$. To show
$\overline{f}$ is int amplified, let $D$ be an ample divisor on $Y$ such
that $f^*D - D$ is also ample (such a divisor $D$ exists since $f$ is
int-amplified). Since $\phi$ is a finite morphism, $\phi^*D$ and
$\phi^*(f^*D - D) = \phi^*f^*D - \phi^*D = \overline{f}^*\phi^*D -
\phi^*D$ are ample. Thus $\overline{f}$ is int-amplified.

\end{proof}

We can use the lifting lemma to show some rank $1$ Gorenstein del Pezzo
surfaces do not admit int-amplified endomorphisms.


\begin{proposition}
The surfaces $S(2A_4)$, $S(D_8)$, $S(A_1+A_7)$, $S(A_8)$ $S(A_2+E_6)$,
$S'(A_2+E_6)$, $S(A_1+E_7)$, and $S'(A_1+E_7)$ do not admit
int-amplified endomorphisms. 
\end{proposition}
\begin{proof}
\cite{Miyanishi-Zhang} describes the singularity types and Picard rank
of the quasi-universal covers of rank $1$ Gorenstein del Pezzo surfaces.
Assume for the sake of contradiction that the surface $S(2A_4)$ admits
an int-amplified endomorphism. Then by \ref{lifting_lemma} it would lift
to an int-amplified endomorphism of the quasi-universal cover of
$S(2A_4)$, which is the smooth quintic del Pezzo surface. By
\cite[Proposition 6.4]{Totaro-Log-Bott} the smooth quintic del Pezzo
surface does not admit an int-amplified endomorphism, yielding a
contradiction: so $S(2A_4)$ cannot admit an int-amplified endomorphism. 

The quasi-universal cover of $S(D_8)$ is a rank $3$ Gorenstein del Pezzo
with singularity type $D_5$. The quasi-universal cover of $S(A_1+A_7)$
is a rank $5$ Gorenstein del Pezzo with singularity type $A_1$. The
quasi-universal cover of $S(A_8)$ is a rank $5$ Gorenstein del Pezzo of
singularity type $A_2$. All these universal covers fail Bott vanishing
for a Cartier divisor by \ref{BakerThm} since for each we have
$10-n-2\rho(X)<0$. Since their universal covers do not admit
int-amplified endomorphisms, $S(D_8)$, $S(A_1+A_7)$ and $S(A_8)$ do not
admit int-amplified endomorphisms. 

The quasi-universal covers of $S(A_2+E_6)$ and $S'(A_2+E_6)$ are
Gorenstein del Pezzo surfaces of rank $3$ with a single $D_4$
singularity. The quasi-universal covers of $S(A_1+E_7)$ and
$S'(A_1+E_7)$ are Gorenstein del Pezzo surfaces of rank $2$ with a
single $E_6$ singularity. Let $Y$ be one of the four surfaces
$S(A_2+E_6)$, $S'(A_2+E_6)$, $S(A_1+E_7)$ or $S'(A_1+E_7)$ and let $\pi:
X \to Y$ be the quasi-universal cover. Since $D_4$ and $E_6$
singularities are not toric, $X$ is not a toric variety. Furthermore, we
know $f^{-1}(Y^0)$ is simply-connected and is $X^0$ minus a finite set
of points: since $X^0$ is smooth, $X^0$ is also simply-connected. Thus
$X$ cannot be a quotient of a toric variety by a finite group acting
freely in codimension one. However, by \cite[Theorem 1.3]{Nakayama}, if
a klt del Pezzo surface of Picard rank $>1$ admits a endomorphism of
degree $>1$, it is a quotient of a toric variety by a finite group
acting freely in codimension one. If $Y$ admitted an int-amplified
endomorphism, it would lift to endomorphism of $X$ of degree $>1$, which
is impossible. Thus the surfaces $S(A_2+E_6)$, $S'(A_2+E_6)$,
$S(A_1+E_7)$ and $S'(A_1+E_7)$ do not admit int-amplified endomorphisms.

\end{proof}

\section{Riemann-Roch for normal surfaces}

The Riemann-Roch formula for normal surfaces generalizes the classical
Hirzebruch-Riemann-Roch formula for vector bundles on a smooth
projective surface and computes the Euler characteristic of a reflexive
sheaf on a surface in terms of intersection theory and local invariants
of singularities. Let us give an overview of this theory, which is
developed in works of Blache \cite{Blache} and Langer
\cite{Langer-Chern}, \cite{Langer}.

Let $X$ be a normal projective variety. A coherent sheaf $\mathcal{E}$
on $X$ is said to be \textit{reflexive} if the natural map $\mathcal{E}
\to \mathcal{E}^{**}$ is an isomorphism. The double-dual of a coherent
sheaf is a reflexive sheaf called its \textit{reflexive hull}. If
$f\colon  X \to Y$ is a morphism and $\mathcal{E}$ a reflexive sheaf on
$X$, we denote $f_{[*]}\mathcal{E} \colonequals (f_*\mathcal{E})^{**}$.
We also denote $\Omega^{[p]}_X \colonequals (\Omega^p_X)^{**}$, the
sheaf of reflexive $p$-forms on $X$, and $\Omega^{[p]}_X(D) \colonequals
(\Omega^p_X \otimes \mathcal{O}(D))^{**}$.

Given a Weil divisor $D$ on $X$, $\mathcal{O}(D)$ is a reflexive sheaf
of rank one on $X$. In the other direction, the vanishing locus of a
nonzero rational section of a reflexive sheaf of rank one is a Weil
divisor. Thus $D \mapsto \mathcal{O}(D)$, $\mathcal{E} \mapsto
c_1(\mathcal{E})$ define a bijection between isomorphism classes of
reflexive rank one sheaves and Weil divisor classes. Define the first
Chern class of a reflexive sheaf $\mathcal{E}$ of rank $r$ by
$c_1(\mathcal{E}) \colonequals c_1( (\wedge^r \mathcal{E})^{**})$.

Now let us require $X$ to be a surface. Let $\pi: \widetilde{X} \to X$ be a
resolution of singularities. Then given any divisor $D$ on $X$, it has a
numerical lift $\pi^*D$ to $\widetilde{X}$, where we can write $\pi^*D =
\overline{D} + D^s$ where $\overline{D}$ is the strict transform of $D$
and the support of $D^s$ is contained in the exceptional locus of $\pi$.
Let $x$ be a singular point of $X$. Given a vector bundle $\mathcal{F}$
on $\widetilde{X}$, there is a local Chern class $c_1(x, \mathcal{F}) \in
\text{CH}^1(\widetilde{X}, \Q)$, supported in $\pi^{-1}(x)$, which satisfies
the formula \[ c_1(\mathcal{F}) = \pi^*c_1(\pi_{[*]}\mathcal{F}) +
\sum_{x \in \text{Sing}X} c_1(x, \mathcal{F}).\]


We use $c_1(x, D)$ as a shorthand for $c_1(x, \mathcal{O}(D))$. If $D =
\pi_{*}\widetilde{D}$, we have
\[c_1(\mathcal{O}(\widetilde{D})) = \pi^*c_1(\mathcal{O}(D)) + \sum_{x \in\text{Sing}X} c_1(x, \widetilde{D}).\]

There is also a local second Chern class $c_2(x, \mathcal{F}) \in
\text{CH}^2(\widetilde{X}, \Q) \isom \Q$, which satisfies many standard
formulas for Chern classes, for example \[ c_2(x, \mathcal{F} \otimes
\mathcal{L}) = c_2(x, \mathcal{F}) + (r-1)c_1(x, \mathcal{F})\cdot
c_1(x, \mathcal{L}) + \frac{r(r-1)}{2}c_1(x, \mathcal{L})^2\] for
$\mathcal{L}$ a line bundle \cite[Proposition 2.4(5)]{Langer-Chern}. We
use the local second Chern class to define the (global) second Chern
class as follows. Suppose $\mathcal{E}$ is a reflexive sheaf on $X$.
Letting $\mathcal{F}$ be any vector bundle on $\widetilde{X}$ such that
$f_{[*]}(\mathcal{F}) = \mathcal{E}$, we set \[ c_2(\mathcal{E})
\colonequals \pi_*c_2(\mathcal{F}) - \sum_{x \in \text{Sing}X}c_2(x,
\mathcal{F})[x].\] 

Finally, there is a local invariant $a(x, \mathcal{E}) \in \R$ which
serves as the correction term for the Riemann-Roch formula. Let us
assume that the singularity at $x$ is a quotient singularity. Then we
have 
\[ a(x, \mathcal{O}(D)) = \frac{1}{2}\{ c_1(x, D)\}\cdot(\lfloor \pi^* D \rfloor - K_{\widetilde{X}})\]
where $\{ \ \}$ is the fractional part \cite[Theorem 1.2(b)]{Blache}.
There are other formulas we can use when $X$ is an ADE singularity
\cite[\textsection 6]{Langer-Chern}. 




We also have \cite[Proposition 5.8]{Langer-Chern}, \[ a(x,
(\Omega^{[1]}\otimes \mathcal{E})^{**}) = a(x, \mathcal{E}) + a(x,
(\mathcal{E} \otimes \omega_X)^{**}) + \frac{\rk \mathcal{E}}{|G|} - k\]
where $k$ is the number of $\mathcal{O}_X$s in a direct sum
decomposition of $\mathcal{E}$ into indecomposable sheaves and $G$ is
the local fundamental group of the singularity at $x$, i.e. locally $(X,
x)$ is isomorphic to $\C^2/G$ with $G$ acting freely in codimension $1$
(for example if there is an $A_n$ singularity at $x$, $G \isom
\Z/(n+1)\Z$).

\begin{theorem}\cite[Theorem 4.4]{Langer} Let $X$ be a normal projective
surface. Then for a reflexive sheaf $\mathcal{E}$ on $X$ we have
\[ \chi(X, \mathcal{E}) = (\rk \mathcal{E}) \chi(X, \mathcal{O}_X) + \frac{1}{2}c_1(\mathcal{E})\cdot (c_1(\mathcal{E})-K_X) - c_2(\mathcal{E}) + \sum_{x \in \Sing X} a(x, \mathcal{E}). \]
\end{theorem}

Here the intersection products are defined by pulling back to a
resolution \`a la Mumford's intersection theory on a normal surface.
This is a generalization of the Hirzebruch-Riemann-Roch theorem for
vector bundles on algebraic surfaces.

\begin{corollary}
Let $X$ be a Gorenstein del Pezzo surface and $D$ a Weil divisor on $X$. Then
\[
\chi(X, \Omega^{[1]}(D)) = 2 + (K_X+2D) \cdot D - c_2(\Omega^{[1]}(D)) + \sum_{x \in \Sing X} a(x, \Omega^{[1]}(D))
\]
\end{corollary}
\begin{proof}
Let $\mathcal{E} = \Omega^{[1]}(D)$. We have $\rk \mathcal{E} = 2$.
Also, $\chi(\mathcal{O}_X) = h^{0,0}-h^{0,1}+h^{0,2} =
h^{0,0}-h^{1,0}+h^{2,0}$ by Hodge symmetry for varieties with quotient
singularities; since $X$ is rational it has no global reflexive
differential forms \cite{Kebekus}, so $h^{1,0}$ and $h^{2,0}$ are both
zero; thus $\chi(X, \mathcal{O}_X) = 1$. Finally $c_1(\Omega^{[1]}(D)) =
c_1((\wedge^2 \Omega^{[1]}(D))^{**}) = c_1(\Omega^{[2]}(2D)) = K_X+2D$.
Plugging this into the above theorem gives us the expression.
\end{proof}


Now we will suppose $D = \pi_*\widetilde{D}$. Then $\Omega^{[1]}_X(D) =
\pi_{[*]}\Omega^1_{\widetilde{X}}(\widetilde{D})$

\[ c_2(\Omega^{[1]}_X(\pi_*D)) = c_2(\Omega^1_{\widetilde{X}}(D)) - \sum_{x
\in \Sing X} c_2(x, \Omega^1_{\widetilde{X}}(\widetilde{D})) =
c_2(\Omega^1_{\widetilde{X}}) + K_{\widetilde{X}} \cdot D + D^2- \sum_{x \in
\Sing X} c_2(x, \Omega^1_{\widetilde{X}}(\widetilde{D})). \] By \cite[Theorem
7.3]{Blache},
\[
c_2(\Omega^{[1]}_X) = 1 + \#\text{Components}(E_x) -\frac{1}{|G|}.
\]

Furthermore,
\[ c_2(x, \Omega^1_{\widetilde{X}}(\widetilde{D})) = c_2(x,
\Omega^1_{\widetilde{X}}) + c_1(x, \Omega^1_{\widetilde{X}}) \cdot c_1(x,
\widetilde{D}) + c_1(x, \widetilde{D})^2\] Since
$c_1(x, \Omega^1_{\widetilde{X}}) = K_{\widetilde{X}}$ and $K_{\widetilde{X}} \cdot
E_i$ for any exceptional curve of $\pi$, the middle term vanishes. Also
$c_2(\Omega^1_{\widetilde{X}}) = c_2(T_{\widetilde{X}})$ which is the Euler
characteristic of $\widetilde{X}$. This is equal to $12$ minus the degree of
$\widetilde{X}$.

Accompanying this paper is a GitHub repository \cite{CodeRepo} which
contains functions to compute $\chi(X, \Omega^{[1]}_X(\pi_*\widetilde{D}))$
for any divisor $\widetilde{D} = \sum a_iC_i$, where $C_i$ are the
$(-1)$-curves on $\widetilde{X}$, in terms of the singularities on $X$,
intersection numbers, and the coefficients $a_i$. Here $X$ is assumed to
be a Gorenstein del Pezzo surface but not necessarily Picard rank $1$. 

We also need to be able to determine if $\pi_*\widetilde{D}$ is ample.
Note that if $X$ is rank 1 and a divisor on $X$ is effective and
nonzero, then it is ample (this follows because any nonzero effective
divisor has positive intersection with an ample divisor). This suffices
in many cases, but in general, we can use the following lemmas to
determine whether $\pi_*(\widetilde{D})$ is ample.

\begin{lemma}\label{weak_dp_cone} Let $\widetilde{X}$ be a weak del Pezzo
surface. If $\widetilde{X}$ is not $\P^2$ or a Hirzebruch surface, then the
closed cone of curves is spanned by $(-1)$- and $(-2)$-curves.
\end{lemma}
\begin{proof}
Since $X$ is weak Fano it is log Fano, so by the Cone Theorem, the cone
of curves of $X$ is spanned by finitely many rational curves. Let $C$ be
an extremal curve. Since $-K_X$ is nef, we have either $K_X\cdot C < 0$
or $K_X\cdot C = 0$. In the case $K_X \cdot C < 0$, by \cite[Theorem
1.28]{KollarMori}, $C$ is a $(-1)$-curve (since we have assumed $X$ is
not isomorphic to $\P^2$ or a $\P^1$-bundle over $\P^1$). In the other
case, we have $K_X \cdot C = 0$. Since $(-K_X)^2 > 0$, the intersection
pairing on $(-K_X)^\perp$ is negative definite, so $C^2<0$. Thus
$(K_X+C)\cdot C = C^2 < 0$, so the arithmetic genus of $C$ must be zero.
Thus $C$ is a smooth rational curve with $C^2 = -2$ as desired.
\end{proof}

\begin{lemma}\label{ample_criterion} Let $X$ be a Gorenstein del Pezzo
surface whose minimal resolution is not $\P^2$ or a Hirzebruch surface
and let $\pi: \widetilde{X} \to X$ be its minimal resolution. Let $C_1,
\dots C_n$ be the $(-1)$-curves on $\widetilde{X}$ and let $D = \sum
a_i(\pi_*C_i)$. Then $D$ is ample on $X$ iff for all $1 \leq i \leq n$, 
\[
\sum_{1 \leq j \leq n} a_j \left(C_i.C_j - \sum_{x \in \Sing X}c_1(x, C_i) \cdot c_1(x, C_j)\right) > 0
\]

\end{lemma}
\begin{proof}
By Lemma \ref{weak_dp_cone}, the cone of curves of $\widetilde{X}$ is
spanned by $(-1)$- and $(-2)$-curves. $\pi$ contracts all $(-2)$-curves,
so the cone of curves of $X$ is spanned by images of $(-1)$-curves. So
any divisor $D$ on $X$ can be written as a linear combination of images
of $(-1)$-curves. Furthermore, $D$ is ample if and only if $(\pi^* D).
C_i > 0$ for $1 \leq i \leq n$ where $C_1, \dots C_n$ are the
$(-1)$-curves on $\widetilde{X}$. 

Let $D = \sum a_j(\pi_*C_j)$. We have $\pi^*\pi_*C_j = C_j - \sum_{x \in
\text{Sing} X} c_1(x, C_j)$. Furthermore, note that for any singular
point $x$, $C_i - c_1(x, C_i)$ does not contain any of the exceptional
curves over $x$, so $(C_i - c_1(x, C_i)) \cdot c_1(x, C_j) = 0$. Thus
$C_i \cdot c_1(x, C_j) = c_1(x, C_j) \cdot c_1(x, C_j)$. So 
\[(\pi^*\pi_*C_i)\cdot C_j = C_i\cdot C_j - \sum_{x \in \text{Sing}X}c_1(x,C_i)\cdot c_1(x, C_j).\]
Thus 
\[
(\pi^*D) \cdot C_i = \sum_{1 \leq j \leq n} a_j \left(C_i \cdot C_j - \sum_{x \in \Sing X}c_1(x, C_i)\cdot c_1(x, C_j)\right)
\]
which gives us the desired result.
\end{proof}

\begin{proposition}
Let $X$ be one of the following surfaces: $S(A_4)$, $S(D_5)$,
$S(A_1+A_5)$, $S(E_6)$, $S(A_7)$, $S(A_1+D_6)$, $S(E_7)$, $S(A_2+A_5)$,
$S(D_8)$, $S(A_1+A_7)$, $S(A_8)$. Then there exists an ample Weil
divisor $D$ on $X$ such that $\chi(X, \Omega^{[1]}(D)) < 0$, so $X$
fails Bott vanishing. In particular, $X$ does not admit an int-amplified
endomorphism. 

\end{proposition}

\begin{proof}
The computation of the Euler characteristic using the Riemann-Roch
formula for all these surfaces is implemented in the code repository
associated to this paper \cite{CodeRepo}. In particular, the code repository
contains, for each surface, the coefficients $a_i$ defining a specific
divisor $\pi_*(\sum a_iC_i)$. We describe the computation in detail
for $S(A_4)$; the rest are similar. 

Let $X = S(A_4)$ and $\pi: \widetilde{X} \to X$ be its minimal resolution.
Then $\widetilde{X}$ is the weak del Pezzo surface 5F in
\cite{Martin-Stadlmayr}. It has one $(-1)$-curve, and the arrangement of
$(-1)$- and $(-2)$-curves looks like
\begin{center}
\begin{tikzpicture}[scale=1.5]
\draw (0,0) -- (2,1);
\draw (1,1) -- (3,0);
\draw (2,0) -- (4,1);
\draw (3,1) -- (5,0);
\draw[dashed] (2.4,1) -- (4.4,0);
\end{tikzpicture}
\end{center}

Here the dashed line represents the $(-1)$-curve and the solid lines
represent $(-2)$-curves. Let $E_1$, $E_2$, $E_3$ and $E_4$ be the
$(-2)$-curves from left to right and let $C \subset \widetilde{X}$ be the
$(-1)$-curve (which meets $E_3$). Set $D = \pi_*C$. The degree of $X$ is
$(-K_X)\cdot (-K_X) = 5$. Also, $(-K_X).D = (-K_{\widetilde{X}}).C = 1$ by
the adjunction formula, since $C$ is a rational curve and $C^2 = -1$.
Since $X$ is Picard rank $1$, $-K_X \equiv 5D$, so $D$ is an
ample Weil divisor on $X$.

$X$ has one singular point $x$ which $D$ meets. To compute the Euler
characteristic $\chi(X, \Omega^{[1]}_X(D))$ with the Riemann-Roch
formula, let us first compute the local invariants $c_2(x,
\Omega^{[1]}_X(D))$ and $a(x, \Omega^{[1]}_X(D))$. First let us compute
the numerical lift $\pi^*D$. We know that $\pi^*D = C + \sum a_iE_i$ and
that $\pi^*D \cdot E_i = 0$. Furthermore $C\cdot E_i = 0$ for $i = 1, 2,
4$ and $C\cdot E_3 = 1$. Solving this system of equations, we obtain 
\[ \pi^*D = C + \frac25 E_1 + \frac 45 E_2 + \frac65 E_3 + \frac35E_4\]
So we have 
\[ 
a(x, \mathcal{O}(D)) = -\frac12\left(\frac25 E_1 + \frac 45 E_2 + \frac15 E_3 + \frac35E_4\right)\cdot (C +  E_3 - K_{\widetilde{X}}) = -\frac12 \left(\frac15 + (\frac45-\frac25+\frac35) \right) = -\frac35.
\]
so $a(x, \Omega^{[1]}(D)) = 2(-\frac35) + \frac15-0 = -1$.

Next, $c_1(x, \mathcal{O}(C))^2 = C^2-D^2 = -1-\frac15(-K_X).C = -\frac65$.
Thus $c_2(x, \Omega^1_{\widetilde{X}}) = 1+4-\frac15-\frac65=\frac{18}{5}$.

Thus $c_2(\Omega^{[1]}(D)) = (12-5)-1-1 -\frac{18}{5}= \frac75$.

Putting it all together we have

\[ \chi(X, \Omega^{[1]}(D)) = 2 + \left(-1 + \frac25\right) - \frac75 -1
= -1 < 0\]

as desired.

\end{proof}

\begin{proposition} Let $Y$ be either $S(2A_1+D_6)$ or $S(A_3+D_5)$.
Then $Y$ does not admit an int-amplified endomorphism.
\end{proposition}
\begin{proof}
First let $Y = S(A_3+D_5)$ first. Then by \cite[p.71]{Miyanishi-Zhang}
the universal cover $X$ of $Y$ is a rank $4$ Gorenstein del Pezzo with
singularity type $A_2$. There is only one such Gorenstein del Pezzo
surface: it satisfies Bott vanishing for Cartier divisors but we show it
fails Bott vanishing for a Weil divisor. Let $\widetilde{X}$ be its minimal
resolution. We can construct $\widetilde{X}$ by blowing up $\P^2$ at four
points, three of which lie on a line, and then blowing up a point $P$ on
the exceptional curve of one of the points on the line such that $P$
doesn't lie on any other $(-1)$-curve. The arrangement of $(-1)-$ and
$(-2)-$curves on $\widetilde{X}$ looks like:
\begin{center}
\includegraphics[scale=0.5]{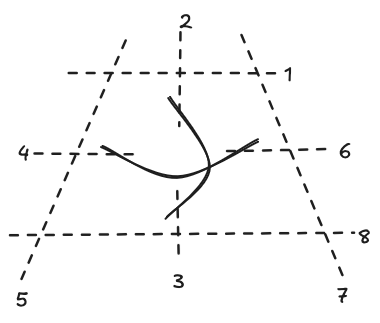}
\end{center}
As before, the dashed lines represent $(-1)$-curves and the solid lines
represent $(-2)$-curves. Now let $\tilde{D} =
2C_1-2C_2-4C_3+C_4+2C_5+3C_6+4C_7+C_8$. This particular divisor was found
through a computer search. Then one can show using Lemma \ref{ample_criterion}
that $D = \pi_*\tilde{D}$ is ample on $X$. Applying the Riemann-Roch formula,
we can compute $\chi(\Omega^{[1]}(D)) < 0$, so $X$ fails Bott vanishing for
$D$. Thus $X$ cannot admit an int-amplified endomorphism. By Proposition
\ref{lifting_lemma} this implies $Y$ does not admit an int-amplified
endomorphism. 

Now let $Y = S(2A_1+D_6)$. The quasi-universal cover $S(2A_1+D_6)$ is of
degree $4$ and has an $A_3$ singularity. There are two such Gorenstein
del Pezzo surfaces: both satisfy Bott vanishing for Cartier divisors
(their minimal resolutions are 4D and 4E in \cite{Martin-Stadlmayr}),
but using the same method, one can show they both fail Bott vanishing
for some Weil divisor. So without knowing which is actually the
universal cover we can nevertheless conclude that $Y$ cannot admit an
int-amplified endomorphism.

All of these specific Riemann-Roch computations can be found in the
associated code repository, as well functions for doing searches for
divisors in the ample cone \cite{CodeRepo}. The repository also
implements Lemma \ref{ample_criterion} for verifying divisors are ample
on Gorenstein del Pezzo surfaces of rank $>1$.

\end{proof}

In summary we have:
\begin{itemize}
    \item $\P^2$, $S(A_1)$, $S(A_1+A_2)$, $S(2A_1+A_3)$ and $S(3A_2)$
    are toric 
    \item $S(3A_1+D_4)$, $S(A_1+2A_3)$, $S(A_1+A_2+A_5)$, $S(4A_2)$ and
    $S(2A_1+2A_3)$ are quotients of toric varieties (of the kind
    described in Theorem \ref{main_theorem})
    \item $S(E_8)$, $S(A_1+E_7)$, $S(A_2+E_6)$ and all surfaces
    $S_t(2D_4)$ fail Bott vanishing for a Cartier divisor, and hence do
    not admit int-amplified endomorphisms
    \item $S(2A_4)$, $S(D_8)$, $S(A_1+A_7)$, $S(A_8)$, $S'(A_2+E_6)$ and
    $S'(A_1+E_7)$ do not admit int-amplified endomorphisms, by lifting
    to universal covers
    \item $S(A_4)$, $S(D_5)$, $S(A_1+A_5)$, $S(E_6)$, $S(A_7)$,
    $S(A_1+D_6)$, $S(E_7)$ and $S(A_2+A_5)$ fail Bott vanishing for a
    Weil divisor, and thus do not admit int-amplified endomorphisms
    \item $S(2A_1+D_6)$ and $S(A_3+D_5)$ have universal covers that fail
    Bott vanishing for a Weil divisor, and thus do not admit
    int-amplified endomorphisms
\end{itemize}

This addresses all surfaces listed in the classification of rank $1$
Gorenstein del Pezzo surfaces except $S'(E_8)$. Thus we conclude Theorem
\ref{main_theorem}. $S'(E_8)$ is not amenable to any of the techniques
in this paper (see also \cite{Gurjar}, which characterizes normal
Gorenstein surfaces dominated by $\P^2$, except for $S'(E_8)$).
$S'(E_8)$ satisfies Bott vanishing for Weil divisors (by \cite{Baker},
it satisfies Bott vanishing for Cartier divisors, and all Weil divisors
are Cartier by \cite[p.71]{Miyanishi-Zhang}) and its smooth locus has
trivial fundamental group, so it has no nontrivial covers. Thus, this
last surface remains open for future investigation.

\bibliographystyle{aomalpha}{}
\bibliography {references}   

\end {document}